\documentclass{article}

\usepackage{PRIMEarxiv}
\usepackage{amsmath,amssymb}

\usepackage{amsthm}

\usepackage[utf8]{inputenc} % allow utf-8 input
\usepackage[T1]{fontenc}    % use 8-bit T1 fonts
\usepackage{hyperref}       % hyperlinks
\usepackage{url}            % simple URL typesetting
\usepackage{booktabs}       % professional-quality tables
\usepackage{amsfonts}       % blackboard math symbols
\usepackage{nicefrac}       % compact symbols for 1/2, etc.
\usepackage{microtype}      % microtypography
\usepackage{lipsum}
\usepackage{fancyhdr}       % header
\usepackage{graphicx}       % graphics
\graphicspath{{media/}}     % organize your images and other figures under media/ folder
\usepackage{float}

\newtheorem{definition}{Definition}
\newtheorem{theorem}{Theorem}

\title{Practical Applications of Unidimensional Optimization Using Octave.
}

\author{
Erick Clapton de Lima Silva\\
\textit{Centro de Ensino Superior do Seridó  }\\
\textit{Federal University of Rio Grande do Norte- UFRN}\\
Caicó, Rio Grande do Norte \\
\texttt{erick.clapton.099@ufrn.edu.br} \\
   \And
  Francisco Marcio Barboza \\
  \textit{Department of Computing and Technology }\\
  \textit{Federal University of Rio Grande do Norte- UFRN}\\
  Caicó, Rio Grande do Norte \\
  \texttt{marcio.barboza@ufrn.br} 
}

\begin{document}
\maketitle

\begin{abstract}
This paper suggests integrating one-dimensional optimization methods to tackle diverse problems, emphasizing their significance in resolving practical issues and applying mathematical principles to real-world contexts. It focuses on employing the Octave programming language, backed by specific examples, to simplify the practical application of mathematical concepts and improve problem-solving abilities. The research aims to assess the effect on students' comprehension of one-dimensional optimization.
\end{abstract}

% keywords can be removed

\section{Introduction}
Unidimensional optimization plays a crucial role in various domains, offering an effective approach to enhancing processes and decision-making. This mathematical discipline focuses on either maximizing or minimizing a single objective function, proving indispensable for addressing a diverse array of practical problems \cite{nocedal2006numerical}.

An integral aspect of unidimensional optimization is its ability to provide optimal solutions in scenarios where complexity can be reduced to a single dimension. By locating the extremes of a function, unidimensional optimization enables us to identify values that either maximize benefits or minimize costs, thereby exerting a positive influence across various aspects of daily life \cite{luenberger1984linear}.

This article explores two specific applications of unidimensional optimization, demonstrating its practical significance: optimizing the distance from the movie screen to enhance viewers' visual experience and calculating the maximum size of an L-shaped pipe to fit through a corridor.

By understanding the significance of these applications, we can discern how unidimensional optimization not only enhances the efficiency and advancement of systems but also contributes significantly to public well-being and health. Throughout this article, we will meticulously examine each application, analyzing the distinct challenges encountered in each domain and showcasing how unidimensional optimization provides efficient solutions.

Optimization applications are predominantly pervasive in computer science and applied mathematics, where they play a crucial role in addressing various problems. However, optimization extends its reach across diverse fields of knowledge, including engineering, natural sciences, and economics \cite{nocedal2006numerical}. Despite the varied disciplines, they share a common objective: leveraging optimization for problem-solving in a more direct and precise manner.

The unidimensional optimization problem holds paramount importance in practical optimization \cite{strang1986introduction}. This significance arises not only from its frequent occurrence in research but also from the fact that more complex problems involving multiple variables can often be tackled through a sequence of one-dimensional problems. In this article, we present some fundamental numerical techniques for locating the minimum of a single-variable function. We focus solely on minimization problems, with the understanding that maximization concerns can be addressed similarly, as minimizing $f(x)$ is equivalent to maximizing $-f(x)$.

\section{Methodology}
 \begin{definition}
  (Function increasing and decreasing) \\
  Be \(f\) a function defined on an interval \(I\), and let \(x_1\) e \(x_2\) any two points in \(I\).
  \begin{enumerate}
    \item \(f\) is increasing on \(I\) if \(x_1 < x_2\) implies that \(f(x_1) < f(x_2)\).
    \item \(f\) is decreasing on \(I\) if \(x_1 < x_2\) implies that \(f(x_1) > f(x_2)\).
  \end{enumerate}
\end{definition}

\begin{theorem}
  Be \(f\) continuous on a closed interval \([a, b]\) and differentiable at \((a, b)\).
  \begin{enumerate}
    \item Se \(f(x) > 0\) for all \(x\) em \((a, b)\), then \(f\) is increasing on \([a, b]\).
    \item Se \(f(x) < 0\) for all \(x\) em \((a, b)\), then \(f\) is decreasing on \([a, b]\). 
  \end{enumerate}
\end{theorem}
\begin{proof}
  See \cite{royden2010real}.
\end{proof}

\begin{definition}
(Local Extrema) \\
Be $c$ belonging to the domain of a function $f$ \\
\begin{enumerate}
    \item $f(c)$ is a local maximum of $f$ If there exists an open interval containing $(a, b)$ including $c$, such that $f(x) \leq f(c)$ for any $x$ em $(a, b)$ that belongs to the domain of $f$.
    \item $f(c)$ is a local minimum of $f$ if there exists an open interval  $(a, b)$  containing $c$, such that $f(x) \geq f(c)$ for any $x$ em $(a, b)$ belonging to the domain of $f$.
\end{enumerate}
\end{definition}

\begin{definition}
(Critical Values) \\
If a function \( f \) has an extreme value at a number \( c \) in an open interval, then either \( f'(c) = 0 \) or \( f'(c) \) does not exist. A number \( c \) with this property is called a critical point of \( f \).
\end{definition}

\begin{definition}
    (First Derivative Test) \\
Suppose $f$ is continuous on an open interval containing a critical value $c$.
\begin{enumerate}
    \item If $f$ changes sign from positive to negative at $c$, then $f$ has a local maximum value at $c$.
    \item If $f$ changes sign from negative to positive at $c$, then $f$ has a local minimum value at $c$.
\end{enumerate}
\end{definition}

\begin{definition}
    (Second Derivative Test) \\
Suppose $f$ is twice different from  $c$.
\begin{enumerate}
    \item Se $f(c) = 0$ e $f'(c) > 0$, then $f$ has a relative minimum in em $x = c$.
    \item Se $f(c) = 0$ e $f'(c) < 0$, then $f$ has a relative maximum in $x = c$.
    \item Se $f(c) = 0$ e $f'(c) = 0$, then the test is inconclusive.
\end{enumerate}
\end{definition}

In many applications, there arises a necessity to ascertain the local minimum or maximum of a function represented as $f(x)$. Definition 5 elucidates that the value of $x$ may correspond to a local minimum or maximum, a determination typically facilitated by finding the root of the function's derivative. The corresponding value of $y$ is then derived by substituting $x$ into the function. In Octave, when seeking the minimum value of a one-dimensional function within the interval $x_1 \leqslant x \leqslant x_2$, one can utilize the command \texttt{fminbnd}.

The \texttt{fminbnd} command in Octave serves to locate the minimum of a one-dimensional function within a predefined interval. Employing an optimization algorithm, it identifies the value of $x$ that minimizes the function $f(x)$ within the specified interval. The fundamental syntax of the command is as follows:

\[
\texttt{[x, fval] = \texttt{fminbnd} (f, x1, x2)}
\]

where:
\begin{itemize}
\item \texttt{f} is the function to be minimized.
\item \texttt{x1} and \texttt{x2} are the bounds of the interval over which the function is evaluated.
\item \texttt{x} is the value of $x$ that minimizes the function within the interval.
\item \texttt{fval} is the minimum value of the function found at \texttt{x}.
\end{itemize}

This command proves particularly handy when swiftly determining the minimum of a function within a designated interval, eliminating the necessity of computing derivatives or resorting to more intricate optimization techniques.

Moreover, the \texttt{fminbnd} command can be employed to pinpoint the maximum of a function. This involves multiplying the function by $-1$ and subsequently seeking the minimum \cite{gilat2004matlab}.

\subsection{Pipe Problem}

A pipe of length $L$ and negligible diameter must be horizontally transported around a corner from a corridor with a width of 3 feet to a corridor with a width of 6 feet (see Figure \ref{Tubulação}). What is the maximum length that the pipe can have? \

\begin{figure} [H]
    \centering
    \includegraphics{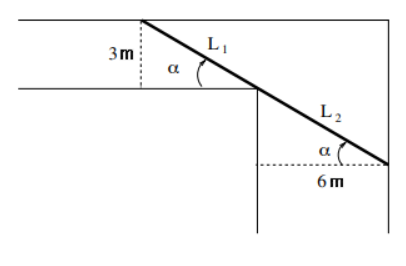}
    \caption{Pipe, Source: \cite{kharab2018introduction}}
    \label{Tubulação}   
\end{figure}
Therefore, $L$ depends on the angle $\alpha$. \

$L(\alpha) = L_1 + L_2 = \frac{3}{\sin(\alpha)} + \frac{6}{\cos(\alpha)}$ \

We conclude that the length of the longest pipeline is given by the absolute minimum of \( L(\alpha) \) in the interval \( (0, \pi/2) \). The graph of \( L(\alpha) \) shows that it has only one relative minimum; therefore, the absolute minimum of \( L(\alpha) \) in \( (0, \pi/2) \) occurs at this relative minimum. The approximate minimum of \( L(\alpha) \) can be obtained using one of the numerical methods described in this paper.

\subsection{Cinema Problem}
To achieve the best viewing experience while watching a movie, it is necessary for a person to position themselves at a distance $x$ from the screen such that the viewing angle $\theta$ is maximized. The problem at hand, taken from \cite{gilat2004matlab}, aims to determine the distance $x$ for which $\theta$ is maximum, as illustrated in the configuration in Figure \ref{Figura Cinema}.
\begin{figure}[H]
    \centering
    \includegraphics{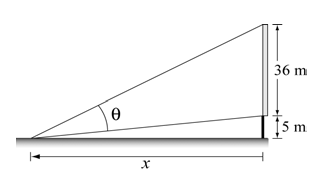}
    \caption{Geometric Modeling of the Problem. Source: \cite{gilat2004matlab}}
    \label{Figura Cinema}
\end{figure}
 \section{Results and discussions}

\subsection{Piping Problem Solution}
This question pertains to determining the maximum length of a pipe that can be horizontally transported around a corner within a corridor. The initial corridor is 3 feet wide, whereas the target corridor is 6 feet wide. To address this problem, numerical methods, such as optimization techniques utilizing the \texttt{fminbnd} command in Octave, can be employed. This method enables the identification of the minimum of a function within a predefined interval, which in this instance spans from $[0, 90]$ degrees. To elaborate on this question, we can reformulate the formula in several stages. \\ \\
$L(\alpha) = L_1 + L_2 = \frac{3}{\sin(\alpha)} + \frac{6}{\cos(\alpha)}$ \\ \\

$L(\alpha) = L_1 + L_2 = 3 \csc(\alpha) + 6 \sec(\alpha)$ \\ \\
This code defines a function  $f(x) = 3 csc(\alpha) + 6 \sec(\alpha)$, where it calculates the minimum of this function in the interval $[0,90]$ using \texttt{fminbnd}, It then prints the minimum and maximum values of the function at this minimum. Finally, the minimum value of this function is maximized in Octave, thus determining the maximum value of the pipe length to pass through the corridor, represented as $L$.
\ref{Figura Cinema}.
\begin{figure}[H]
    \centering
    \includegraphics[width=0.8\linewidth]{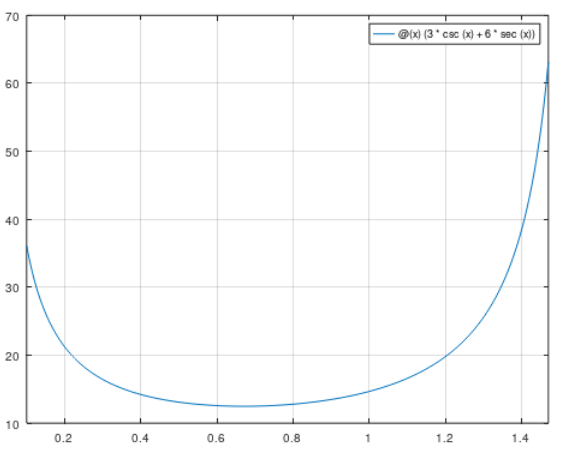}
    \caption{Graph of the equation. Source: \cite{gilat2004matlab}}
    \label{Grafico Tubulação}
\end{figure}

Analyzing the function's graph, based on the values we took for the width of the L-shaped corridor, 3 feet and 6 feet respectively, when running the code, we arrived at the value of approximately $64.403$. In this case, using the method where the maximum of the function would be calculated as the minimum with opposite sign, as explained earlier, $\max = -\min$, we were able to calculate the maximum length that the metal beam can reach to pass through the width of the corridor.

\subsection{Cinema Problem Solution}
To solve this problem using Octave, we can employ trigonometric calculations to derive the expression for the viewing angle $\theta$ in terms of the distance $x$. Subsequently, we aim to maximize this expression to determine the distance $x$ that yields the maximum$ \theta$.

According to Figure \ref{Figura Cinema}, the configuration is a right triangle where the distance $x$ is the hypotenuse, and the other two sides are the height $h$ (distance between the screen and the eyes) and the base $d$ (width of the screen).

Trigonometry tells us that:

\[
\cos(\theta) = \frac{x^2 - \frac{5}{2}^2}{\left(x^2 + \frac{4}{12}^2\right) + 36} - \frac{2}{x^2 + \frac{5}{2}^2 + x^2 + \frac{4}{12}^2}
\]
That is:
\[
\theta = \arccos \left ( \dfrac{h}{d} \right)
\]
Therefore, we can write $h$ in terms of $x$ and $d$ using the Pythagorean theorem:

\[
h = \sqrt{x^2 - \left(\frac{d}{2}\right)^2}
\]

Substituting $h$ in the tangent equation, we obtain:

\[
\theta(x) = \arccos \left(\frac{\sqrt{x^2 - \left(\frac{d}{2}\right)^2}}{d} \right)
\]

Assigning the values given in the problem, where by the Pythagorean theorem the opposite leg, in this case denoted as $d$ with a value of 40 meters, and additionally assigning a value to $x$ to calculate the value of $\theta$ where the value would be 10 meters, upon running the code, we obtain a result of 1.57, which corresponds to the angle of inclination. Furthermore, the angle at which the individual would be seated in the cinema can be calculated by applying the corresponding values to where the person is fixed in the cinema, thus determining the cinema's inclination angle with respect to $\theta$.

 \section{Conclusion}

Learning about optimization in issues involving maxima and minima of functions is incredibly valuable, not only for the development of mathematical skills but also for acquiring analytical and practical skills that are essential in various areas of knowledge. Octave is an extremely effective tool that not only facilitates teaching but also enhances the connection with the content, simplifying the calculation of maxima and minima of a function. This work represents a way to simplify the understanding of the topic of maxima and minima, highlighting the importance of optimization and the use of Octave

%Bibliography
\bibliographystyle{unsrt}  
\bibliography{references}

\end{document}